\documentclass[12pt]{amsart}       
\usepackage{amssymb}
\usepackage{eucal}
\usepackage{graphicx}
\usepackage{amsmath}
\usepackage{amscd}
\usepackage[all]{xy}           
\usepackage{tikz}
\usepackage{amsfonts,latexsym}
\usepackage{xspace}
\usepackage{epsfig}
\usepackage{float}
\usepackage{color}
\usepackage{colordvi}
\usepackage{multicol}
\usepackage{colordvi}
\usepackage[colorlinks,final,backref=page,hyperindex,hypertex]{hyperref}

\newtheorem{theorem}{Theorem}[section]
\theoremstyle{definition}
\newtheorem{defn}[theorem]{Definition}
\newtheorem{lemma}[theorem]{Lemma}

\newtheorem{prop-def}{Proposition-Definition}[section]
\newtheorem{coro-def}{Corollary-Definition}[section]
\newtheorem{claim}[theorem]{Claim}
\newtheorem{remark}[theorem]{Remark}

\newcommand{\nc}{\newcommand}
\nc{\tred}[1]{\textcolor{red}{#1}}
\nc{\tblue}[1]{\textcolor{blue}{#1}}
\nc{\tgreen}[1]{\textcolor{green}{#1}}
\nc{\tpurple}[1]{\textcolor{purple}{#1}}
\nc{\btred}[1]{\textcolor{red}{\bf #1}}
\nc{\btblue}[1]{\textcolor{blue}{\bf #1}}
\nc{\btgreen}[1]{\textcolor{green}{\bf #1}}
\nc{\btpurple}[1]{\textcolor{purple}{\bf #1}}
\nc{\NN}{{\mathbb N}}
\nc{\ncsha}{{\mbox{\cyr X}^{\mathrm NC}}} \nc{\ncshao}{{\mbox{\cyr
X}^{\mathrm NC}_0}}

\renewcommand{\frak}{\mathfrak}

\newcommand{\efootnote}[1]{}

\renewcommand{\textbf}[1]{}

\newcommand{\delete}[1]{}

\nc{\mlabel}[1]{\label{#1}}  
\nc{\mcite}[1]{\cite{#1}}  
\nc{\mref}[1]{\ref{#1}}  
\nc{\mbibitem}[1]{\bibitem{#1}} 

\delete{
\nc{\mlabel}[1]{\label{#1}  
{\hfill \hspace{1cm}{\small\tt{{\ }\hfill(#1)}}}}
\nc{\mcite}[1]{\cite{#1}{\small{\tt{{\ }(#1)}}}}  
\nc{\mref}[1]{\ref{#1}{{\tt{{\ }(#1)}}}}  
\nc{\mbibitem}[1]{\bibitem[\bf #1]{#1}} 
}

\nc{\li}[1]{\textcolor{red}{Li:#1}}
\nc{\yi}[1]{\textcolor{blue}{Yi: #1}}
\nc{\xing}[1]{\textcolor{purple}{Xing:#1}}
\nc{\revise}[1]{\textcolor{red}{#1}}

\nc{\mo}{{\rm M}}
\nc{\barot}{\overline{\otimes}}

\nc{\opa}{\ast} \nc{\opb}{\odot} \nc{\op}{\bullet} \nc{\pa}{\frakL}
\nc{\arr}{\rightarrow} \nc{\lu}[1]{(#1)} \nc{\mult}{\mrm{mult}}
\nc{\diff}{\mathfrak{Diff}}
\nc{\opc}{\sharp}\nc{\opd}{\natural}
\nc{\ope}{\circ}
\nc{\dpt}{\mathrm{d}}
\nc{\diam}{alternating\xspace}
\nc{\Diam}{Alternating\xspace}
\nc{\cdiam}{canonical alternating\xspace}
\nc{\Cdiam}{Canonical alternating\xspace}
\nc{\AW}{\mathcal{A}}
\nc{\mrbo}{modified RBO\xspace }
\nc{\ari}{\mathrm{ar}}

\nc{\lef}{\mathrm{lef}}

\nc{\Sh}{\mathrm{ST}}

\nc{\Cr}{\mathrm{Cr}}

\nc{\st}{{Schr\"oder tree}\xspace}
\nc{\sts}{{Schr\"oder trees}\xspace}

\nc{\vertset}{\Omega} 

\nc{\pb}{{\mathrm{pb}}}
\nc{\Lf}{{\mathrm{Lf}}}

\nc{\lft}{{left tree}\xspace}
\nc{\lfts}{{left trees}\xspace}

\nc{\fat}{{fundamental averaging tree}\xspace}

\nc{\fats}{{fundamental averaging trees}\xspace}
\nc{\avt}{\mathrm{Avt}}

\nc{\rass}{{\mathit{RAss}}}

\nc{\aass}{{\mathit{AAss}}}

\nc{\vin}{{\mathrm Vin}}    
\nc{\lin}{{\mathrm Lin}}    
\nc{\inv}{\mathrm{I}n}
\nc{\gensp}{V} 
\nc{\genbas}{\mathcal{V}} 
\nc{\bvp}{V_P}     
\nc{\gop}{{\,\omega\,}}     

\nc{\bin}[2]{ (_{\stackrel{\scs{#1}}{\scs{#2}}})}  
\nc{\binc}[2]{ \left (\!\! \begin{array}{c} \scs{#1}\\
    \scs{#2} \end{array}\!\! \right )}  
\nc{\bincc}[2]{  \left ( {\scs{#1} \atop
    \vspace{-1cm}\scs{#2}} \right )}  
\nc{\bs}{\bar{S}} \nc{\cosum}{\sqsubset} \nc{\la}{\longrightarrow}
\nc{\rar}{\rightarrow} \nc{\dar}{\downarrow} \nc{\dprod}{**}
\nc{\dap}[1]{\downarrow \rlap{$\scriptstyle{#1}$}}
\nc{\md}{\mathrm{dth}} \nc{\uap}[1]{\uparrow
\rlap{$\scriptstyle{#1}$}} \nc{\defeq}{\stackrel{\rm def}{=}}
\nc{\disp}[1]{\displaystyle{#1}} \nc{\dotcup}{\
\displaystyle{\bigcup^\bullet}\ } \nc{\gzeta}{\bar{\zeta}}
\nc{\hcm}{\ \hat{,}\ } \nc{\hts}{\hat{\otimes}}
\nc{\free}[1]{\bar{#1}}
\nc{\uni}[1]{\tilde{#1}} \nc{\hcirc}{\hat{\circ}} \nc{\lleft}{[}
\nc{\lright}{]} \nc{\lc}{\lfloor} \nc{\rc}{\rfloor}
\nc{\curlyl}{\left \{ \begin{array}{c} {} \\ {} \end{array}
    \right .  \!\!\!\!\!\!\!}
\nc{\curlyr}{ \!\!\!\!\!\!\!
    \left . \begin{array}{c} {} \\ {} \end{array}
    \right \} }
\nc{\longmid}{\left | \begin{array}{c} {} \\ {} \end{array}
    \right . \!\!\!\!\!\!\!}
\nc{\onetree}{\bullet} \nc{\ora}[1]{\stackrel{#1}{\rar}}
\nc{\ola}[1]{\stackrel{#1}{\la}}
\nc{\ot}{\otimes} \nc{\mot}{{{\boxtimes\,}}}
\nc{\otm}{\overline{\boxtimes}} \nc{\sprod}{\bullet}
\nc{\scs}[1]{\scriptstyle{#1}} \nc{\mrm}[1]{{\rm #1}}
\nc{\margin}[1]{\marginpar{\rm #1}}   
\nc{\dirlim}{\displaystyle{\lim_{\longrightarrow}}\,}
\nc{\invlim}{\displaystyle{\lim_{\longleftarrow}}\,}
\nc{\mvp}{\vspace{0.3cm}} \nc{\tk}{^{(k)}} \nc{\tp}{^\prime}
\nc{\ttp}{^{\prime\prime}} \nc{\svp}{\vspace{2cm}}
\nc{\vp}{\vspace{8cm}} \nc{\proofbegin}{\noindent{\bf Proof: }}
\nc{\proofend}{$\blacksquare$ \vspace{0.3cm}}
\nc{\modg}[1]{\!<\!\!{#1}\!\!>}
\nc{\intg}[1]{F_C(#1)} \nc{\lmodg}{\!
<\!\!} \nc{\rmodg}{\!\!>\!}
\nc{\cpi}{\widehat{\Pi}}
\nc{\sha}{{\mbox{\cyr X}}}  
\nc{\ssha}{{\mbox{\cyrs X}}} 
\nc{\shpr}{\diamond}    
\nc{\shp}{\ast} \nc{\shplus}{\shpr^+}
\nc{\shprc}{\shpr_c}    
\nc{\msh}{\ast} \nc{\zprod}{m_0} \nc{\oprod}{m_1}
\nc{\labs}{\mid\!} \nc{\rabs}{\!\mid}
\nc{\sqmon}[1]{\langle #1\rangle}

\nc{\mmbox}[1]{\mbox{\ #1\ }} \nc{\dep}{\mrm{dep}} \nc{\fp}{\mrm{FP}}
\nc{\rchar}{\mrm{char}} \nc{\End}{\mrm{End}} \nc{\Fil}{\mrm{Fil}}
\nc{\Mor}{Mor\xspace} \nc{\gmzvs}{gMZV\xspace}
\nc{\gmzv}{gMZV\xspace} \nc{\mzv}{MZV\xspace}
\nc{\mzvs}{MZVs\xspace} \nc{\Hom}{\mrm{Hom}} \nc{\id}{\mrm{id}}
\nc{\im}{\mrm{im}} \nc{\incl}{\mrm{incl}} \nc{\map}{\mrm{Map}}
\nc{\mchar}{\rm char} \nc{\nz}{\rm NZ} \nc{\supp}{\mathrm Supp}

\nc{\Alg}{\mathbf{Alg}} \nc{\Bax}{\mathbf{Bax}} \nc{\bff}{\mathbf f}
\nc{\bfk}{{\bf k}} \nc{\bfone}{{\bf 1}} \nc{\bfx}{\mathbf x}
\nc{\bfy}{\mathbf y}
\nc{\base}[1]{\bfone^{\otimes ({#1}+1)}} 
\nc{\Cat}{\mathbf{Cat}}

\nc{\detail}{\marginpar{\bf More detail}
    \noindent{\bf Need more detail!}
    \svp}
\nc{\Int}{\mathbf{Int}} \nc{\Mon}{\mathbf{Mon}}
\nc{\rbtm}{{shuffle }} \nc{\rbto}{{Rota-Baxter }}
\nc{\remarks}{\noindent{\bf Remarks: }} \nc{\Rings}{\mathbf{Rings}}
\nc{\Sets}{\mathbf{Sets}} \nc{\wtot}{\widetilde{\odot}}
\nc{\wast}{\widetilde{\ast}} \nc{\bodot}{\bar{\odot}}
\nc{\bast}{\bar{\ast}} \nc{\hodot}[1]{\odot^{#1}}
\nc{\hast}[1]{\ast^{#1}} \nc{\mal}{\mathcal{O}}
\nc{\tet}{\tilde{\ast}} \nc{\teot}{\tilde{\odot}}
\nc{\oex}{\overline{x}} \nc{\oey}{\overline{y}}
\nc{\oez}{\overline{z}} \nc{\oef}{\overline{f}}
\nc{\oea}{\overline{a}} \nc{\oeb}{\overline{b}}
\nc{\weast}[1]{\widetilde{\ast}^{#1}}
\nc{\weodot}[1]{\widetilde{\odot}^{#1}} \nc{\hstar}[1]{\star^{#1}}
\nc{\lae}{\langle} \nc{\rae}{\rangle}
\nc{\lf}{\lfloor}
\nc{\rf}{\rfloor}

\nc{\QQ}{{\mathbb Q}}
\nc{\RR}{{\mathbb R}} \nc{\ZZ}{{\mathbb Z}}

\nc{\cala}{{\mathcal A}} \nc{\calb}{{\mathcal B}}
\nc{\calc}{{\mathcal C}}
\nc{\cald}{{\mathcal D}} \nc{\cale}{{\mathcal E}}
\nc{\calf}{{\mathcal F}} \nc{\calg}{{\mathcal G}}
\nc{\calh}{{\mathcal H}} \nc{\cali}{{\mathcal I}}
\nc{\call}{{\mathcal L}} \nc{\calm}{{\mathcal M}}
\nc{\caln}{{\mathcal N}} \nc{\calo}{{\mathcal O}}
\nc{\calp}{{\mathcal P}} \nc{\calr}{{\mathcal R}}
\nc{\cals}{{\mathcal S}} \nc{\calt}{{\mathcal T}}
\nc{\calu}{{\mathcal U}} \nc{\calw}{{\mathcal W}} \nc{\calk}{{\mathcal K}}
\nc{\calx}{{\mathcal X}} \nc{\CA}{\mathcal{A}}

\nc{\fraka}{{\mathfrak a}} \nc{\frakA}{{\mathfrak A}}
\nc{\frakb}{{\mathfrak b}} \nc{\frakB}{{\mathfrak B}}
\nc{\frakD}{{\mathfrak D}} \nc{\frakF}{\mathfrak{F}}
\nc{\frakf}{{\mathfrak f}} \nc{\frakg}{{\mathfrak g}}
\nc{\frakH}{{\mathfrak H}} \nc{\frakL}{{\mathfrak L}}
\nc{\frakM}{{\mathfrak M}} \nc{\bfrakM}{\overline{\frakM}}
\nc{\frakm}{{\mathfrak m}} \nc{\frakP}{{\mathfrak P}}
\nc{\frakN}{{\mathfrak N}} \nc{\frakp}{{\mathfrak p}}
\nc{\frakS}{{\mathfrak S}} \nc{\frakT}{\mathfrak{T}}
\nc{\frakX}{\mathfrak{X}} \nc{\frakx}{\mathfrak{x}}
\nc{\fraky}{\mathfrak{y}}

\nc{\BS}{\mathbb{S}}

\font\cyr=wncyr10 \font\cyrs=wncyr7
\nc{\xigou}[1]{\textcolor{blue}{Xigou: #1}}
\nc{\UN}{U_{N}} \nc{\lbar}[1]{\overline{#1}}
\nc{\FN}{F_{\kappa}} \nc{\FNN}{F_{-\lambda^2}}
\nc{\bre}{{\rm bre}} \nc{\w}{{\rm bre}}
\nc{\altx}{\Lambda}
\nc{\spr}{\cdot} \nc{\hma}{\mathcal{H}}
\nc{\rts}{\stackrel{\rightarrow}{\shpr}}
\nc{\ox}{\overline{\frak x}}
\nc{\oX}{\overline{X}} \nc{\mt}{T}
\nc{\mtw}{{\rm MTW}} \nc{\vep}{\varepsilon}
\nc{\mop}{P_A} \nc{\aop}{\alpha} \nc{\apr}{\cdot}
\nc{\otp}{Q} \nc{\ottp}{\tilde{Q}}

\begin{document}

\title[Modified Rota-Baxter algebras and Hopf algebras]{Free modified Rota-Baxter algebras and Hopf algebras}
%
\author{Xigou Zhang}
\address{Department of Mathematics, Jiangxi Normal University, Nanchang, Jiangxi 330022, China}
\email{xyzhang@jxnu.edu.cn}

\author{Xing Gao}
\address{School of Mathematics and Statistics, Key Laboratory of Applied Mathematics and Complex Systems, Lanzhou University, Lanzhou, Gansu 730000, P.\,R. China}
\email{gaoxing@lzu.edu.cn}

\author{Li Guo}
\address{Department of Mathematics and Computer Science,
         Rutgers University,
         Newark, NJ 07102, USA}
\email{liguo@rutgers.edu}

\date{\today}
\begin{abstract}
The notion of a modified Rota-Baxter algebra comes from the combination of those of a Rota-Baxter algebra and a modified Yang-Baxter equation. In this paper, we first construct free modified Rota-Baxter algebras. We then equip a free modified Rota-Baxter algebra with a bialgebra structure by a cocycle construction. Under the assumption that the generating algebra is a connected bialgebra, we further equip the free modified Rota-Baxter algebra with a Hopf algebra structure.
\end{abstract}

\keywords{Modified Rota-Baxter algebra, Rota-Baxter algebra, Hopf algebra, bracketed word, cocycle}

\subjclass[2010]{16T99,16W99,16S10}

\maketitle

\tableofcontents

\setcounter{section}{0}

\allowdisplaybreaks


\section{Introduction}

This paper studies free objects in the category of modified Rota-Baxter algebras, a concept coming from the combination of a Rota-Baxter algebra and a modified Yang-Baxter equation. It also equips the free objects with bialgebra and Hopf algebra structures.

For a fixed constant $\lambda$, a {\bf Rota-Baxter operator} of weight $\lambda$ is a linear operator $P$ on an associative algebra $R$ that satisfies the {\bf Rota-Baxter equation}:
\begin{equation}
 P(x)P(y)=P(P(x)y) + P(xP(y)) +\lambda P(xy), \quad \forall x, y\in R.
 \mlabel{eq:rbe}
 \end{equation}
An associative algebra $R$ equipped with a Rota-Baxter operator is called a {\bf Rota-Baxter algebra}, a notion originated from the probability study of G. Baxter~\mcite{Ba} in 1960. Later it attracted the attention of well-known mathematicians such as Atkinson, Cartier and Rota~\mcite{At,Ca,Ro}. After some years of dormancy, its study experienced a quite remarkable renascence since late 1990s, with many applications in mathematics and physics~\mcite{Ag,BBGN,EG,GK1,JZ,MN,ML,MY,PBG}. In particular, it appeared as one of the fundamental algebraic structures in the profound work of Connes and Kreimer on renormalization of quantum field theory~\mcite{CK}. See~\mcite{Gub} for further details and references.

The concept of the classical Yang-Baxter equation arose from the study of inverse scattering theory and is also related to Schouten bracket in differential geometry. Further it can be regarded as the classical limit of the quantum Yang-Baxter equation, named after C.-Y. Yang and R. Baxter. In the 1980s, Semonov-Tian-Shansky~\cite{Sem} found that,  under suitable conditions, the operator form of the classical Yang-Baxter equation is precisely the Rota-Baxter identity \eqref{eq:rbe} (of weight 0) on a Lie algebra. As a modified form of the operator form of the classical Yang-Baxter equation, he also introduced in that paper the {\bf modified classical Yang-Baxter equation}:
\begin{equation}
[P(x),P(y)]=P[P(x),y]+P[x,P(y)]- [x, y],
\mlabel{eq:mybel}
\end{equation}
later found applications in the study of generalized Lax pairs and affine geometry on Lie groups~\mcite{BGN1,Bo,KS}.
As the associative analogue of Eq.~(\mref{eq:mybel}), the equation
\begin{equation}
P(x)  P(y) = P (P(x)y) + P(xP(y))- xy. \mlabel{eq:mybea}
\end{equation}
is called the {\bf modified associative Yang-Baxter equation}, which has been applied to the study of extended $\calo$-operators, associative Yang-Baxter equations, infinitesimal bialgebras and dendriform algebras~\mcite{BGN2,BGN3,E}.

In the spirit of the aforementioned Yang-Baxter equation to Rota-Baxter operator connection, a linear operator $P$ satisfying Eq.~(\mref{eq:mybea}) is called a {\bf modified Rota-Baxter operator} and an associative algebra $R$ equipped with a modified Rota-Baxter operator is called a {\bf modified Rota-Baxter algebra}.

Integrating the notions of the Rota-Baxter algebra and modified Rota-Baxter algebra, the concept of a modified Rota-Baxter algebra with a weight was introduced in~\mcite{BGN2} as a special case of extended $\mathcal{O}$-operators in connection with the extended associative Yang-Baxter equation. The latter motivated their study in the Lie algebra context~\mcite{BGN4}.
In~\mcite{ZGG}, free commutative modified Rota-Baxter algebras were constructed by means of a modified quasi-shuffle product and modified stuffle product, in analogy to the case of free commutative Rota-Baxter algebras~\mcite{Ca,GK1}.

Considering the close relationship between the modified Rota-Baxter (associative) algebras and the modified Yang-Baxter equation for Lie algebras, it is especially interesting to consider noncommutative modified Rota-Baxter algebras. This is the subject of study of this paper, focusing on the construction of the free objects and the Hopf algebra structures on the free objects. More precisely, in Section~\mref{sec:free}, we obtain an explicit construction of the free modified Rota-Baxter algebra on an algebra, by giving a natural basis of the algebra and the corresponding multiplication table. In Section~\mref{sec:hopf}, we further provide a bialgebra and then a Hopf algebra structure on the free modified Rota-Baxter algebra.

\vskip0.1in

\noindent
{\bf Notations.} For the rest of this paper, unless otherwise specified, algebras are {\bf associative unitary} algebras over a commutative unitary algebra $\bfk$.

\section{Free Modified Rota-Baxter Algebras}
\mlabel{sec:free}
In this section we construct free modified Rota-Baxter algebras. We give the construction in Section~\mref{ss:bmrb}, leading to the main Theorem~\mref{thm:freeao} of this section. The proof of the theorem is completed in Section~\mref{ss:proof}.

\subsection{The general construction of the free modified Rota-Baxter algebras}
\mlabel{ss:bmrb}

We begin with the general definition of modified Rota-Baxter algebras.

\begin{defn}
Let $R$ be a $\bfk$-algebra and $\kappa\in \bfk$. A linear map $P:R\to R$ is called a {\bf modified Rota-Baxter operator} of weight $\kappa$ if $P$ satisfies the operator identity
\begin{equation}
P(u)P(v)=P(uP(v))+P(P(u)v)+\kappa uv, \quad \text{for all } u, v\in R.
\mlabel{eq:mrbo}
\end{equation}
Then the pair $(R, P)$ or simply $R$ is called a {\bf modified Rota-Baxter algebra} of weight $\kappa$.
\mlabel{de:mrba}
\end{defn}

Together with the algebra homomorphisms between the algebras that preserves the linear operators, the class of modified Rota-Baxter algebras of weight $\kappa$ forms a category. We refer the reader to~\mcite{ZGG} and the references therein for basic properties of modified Rota-Baxter algebras and focus our attention to the construction of free modified Rota-Baxter algebras. We first give the definition.

\begin{defn}
Let $A$ be a $\bfk$-algebra. A {\bf free modified Rota-Baxter algebra on $A$} is a modified Rota-Baxter algebra $(F(A),P_A)$ together with an algebra homomorphism $j:A\longrightarrow F(A)$ with the property that, for any given modified Rota-Baxter algebra $(R,P)$ and algebra homomorphism $f:A \longrightarrow R$, there is a unique homomorphism $\free{f}: F(A)\longrightarrow R$ of modified Rota-Baxter algebras such that $\free{f} j =f$.
\mlabel{de:fmrba}
\end{defn}

Note that taking $A$ to be the free algebra $\bfk\langle Y \rangle$ on a set $Y$, we obtain the free modified Rota-Baxter algebra on the set $Y$.
Let $A$ be a  $\bfk$-algebra with a
$\bfk$-basis $X$. We first display a $\bfk$-basis $\frak X_\infty$ of
free modified Rota-Baxter algebras in terms of bracketed words from the alphabet set $X$.

\begin{remark}
The set $\frak X_\infty$ is called the set of {\bf Rota-Baxter words} that was applied to construct free Rota-Baxter algebras~\mcite{EG}. Enumeration properties and generating functions of Rota-Baxter words were obtained in~\mcite{GS} to which we refer the reader for further details.
\mlabel{rk:rbw}
\end{remark}

Let $\lc$ and $\rc$ be two different symbols not in $X$, called brackets,
and let $X':= X\cup \{\lc,\rc\}$. Denote by $M(X')$ the free monoid generated by $X'$.

\begin{defn} (\mcite{EG2,Gub})
Let $Y,Z$ be two subsets of $M(X')$. Define the {\bf alternating
product} of $Y$ and $Z$ to be \allowdisplaybreaks{
\begin{eqnarray*}
\altx(Y,Z)& = &\Big( \bigsqcup_{r\geq 1} \big (Y\lc Z\rc \big)^r \Big)
\bigsqcup
    \Big(\bigsqcup_{r\geq 0} \big (Y\lc Z\rc \big)^r  Y\Big)
    \bigsqcup \Big( \bigsqcup_{r\geq 1} \big( \lc Z\rc Y \big )^r \Big)
 \bigsqcup \Big( \bigsqcup_{r\geq 0} \big (\lc Z\rc Y\big )^r \lc Z\rc \Big).
\end{eqnarray*}}
Here $\sqcup$ stands for disjoint union.
\end{defn}

For example, $y_1\lc z_1\rc y_2, \lc z_1\rc y_1 \lc z_2\rc, y_1, y_2\in Y, z_1, z_2\in Z,$ are elements in $\altx(Y,Z)$. But $\lc z_1\rc \lc z_2\rc$ are not in $\altx(Y,Z)$.

We construct a sequence $\frak X_n$ of subsets of $M(X')$ by the
following recursion on $n\geq 0$. For the initial step, we define $\frak X_0:=X\cup \{1\}$.
For the inductive step, we define
\begin{equation*}
\frak X_{n+1} := \altx(X,\frak X_n) \cup \{1\}\, \text{ for } n\geq 1.
\end{equation*}

For example, for $x_1,x_2,x_3\in X$, the elements $\lc x_1\rc x_2 \lc x_3\rc, x_1\lc x_2\lc x_3\rc \rc$ and $\lc \lc x_1 \rc x_2 \lc\lc x_3\rc\rc\rc$ are all in $\frak X_3$, the first two are in $\frak X_2$ and the first one is in $\frak X_1$.

From the definition we have
$\frak X_1\supseteq \frak X_0$. Assuming $\frak X_n\supseteq \frak X_{n-1}$, we get
$$\frak X_{n+1}=\altx(X,\frak X_n) \supseteq \altx(X,\frak X_{n-1}) =\frak X_n.$$
Thus we can define $$\frakX_\infty:= \dirlim \frakX_n  = \bigcup_{n\geq 0} \frak X_n.$$
For $\frakx\in \frakX_\infty$, we define the {\bf depth} $\dep(\frakx)$ of $\frakx$
to be $$\dep(\frakx):= \min\{ n\mid \frakx \in \frakX_n\}.$$
Further, every $\frak x\in \frak X_\infty\setminus\{1\}$ has a unique {\bf
standard decomposition:}
\begin{equation}
 \frak x=\frak x_1 \cdots  \frak x_b,
\mlabel{eq:st}
\end{equation}
where $\frak x_i$, $1\leq i\leq b$, are alternatively in $X$ or in
$\lc \frak X_\infty\rc$.
We call $b$ to be the {\bf breadth} of $\frak x$, denoted by $\bre(\frak x)$.
We define the {\bf head} $h(\frak x)$ of $\frak x$ to be 0 (resp. 1) if $\frak x_1$ is in $X$ (resp. in $\lc \frak X_\infty \rc$). Similarly define the {\bf tail} $t(\frak x)$ of
$\frak x$ to be 0 (resp. 1) if $\frak x_b$ is in $X$ (resp. in $\lc
\frak X_\infty \rc$).

Fix a $\kappa\in \bfk$. We will equip the free $\bfk$-module
\begin{equation}
\FN(A)=\bfk\,\frakX_\infty=\bigoplus_{\frak x\in \frak X_\infty} \bfk \frak x
\mlabel{eq:fma}
\end{equation}
with a multiplication $\shpr:=\shpr_\kappa$. This is accomplished by defining $\frak x\shpr
\frak x'\in \FN(A)$ for basis elements $\frak x,\frak x'\in \frak X_\infty$ and then
extending bilinearly. Roughly speaking, the product of $\frak x$ and
$\frak x'$ is defined to be the concatenation whenever $t(\frak x)\neq
h(\frak x')$. When $t(\frak x)=h(\frak x')$, the product is defined by
the product in $A$ or by the modified Rota-Baxter identity in
Eq.~(\mref{eq:mrbo}).

To be precise, we use induction on the sum $n:= \dep(\frak x)+\dep(\frak x')\geq 0$ to define
$\frakx\diamond \frakx'$. For the initial step of $n=0$,
$\frak x,\frak x'$ are in $X$ and so are in $A$. Then we define
$$\frak x\shpr \frak x':=\frak x \spr \frak x'\in A \subseteq \FN(A).$$
Here $\spr$ is the product in $A$.

For the inductive step, let $k\geq 0$ be given and assume that $\frak x\shpr \frak x'$ have been defined for all
$\frak x,\frak x'\in \frak X_\infty$ with $n = \dep(\frak x)+\dep(\frak x') \leq k$. Then consider $\frak x, \frak x'\in \frak X_\infty$ with $n= \dep(\frak x)+\dep(\frak x')=k+1$.
First treat the case when $\bre(\frak x)=\bre(\frak x')=1$. Then $\frak x$ and
$\frak x'$ are in $X$ or $\lc \frak X_\infty\rc$. Since $n=k+1\geq 1$, $\frak x$ and $\frak x'$ cannot be both in $X$. We accordingly define
\begin{equation}
\frak x \shpr \frak x' :=\left \{ \begin{array}{ll}
\frak x \frak x', & {\rm if\ } \frak x\in X\,\text{ and }\, \frak x'\in \lc \frak X_\infty\rc,\\
\frak x \frak x', & {\rm if\ } \frak x\in \lc \frak X_\infty\rc\,\text{ and }\, \frak x'\in X,\\
\lc \lc \ox\rc \shpr \ox'\rc +\lc \ox \shpr \lc \ox'\rc \rc +\kappa
\ox \shpr \ox', & {\rm if\ } \frak x=\lc \ox\rc\,\text{ and }\, \frak x'=\lc
\ox'\rc \in \lc \frak X_\infty \rc.
\end{array} \right .
\mlabel{eq:shprod}
\end{equation}
Here the product in the first and second case are by concatenation and in the third case is
by the induction hypothesis since for the three products on the
right hand side we have
\begin{eqnarray*}
\dep(\lc\ox \rc)+ \dep(\ox') &=& \dep(\lc \ox \rc)+\dep(\lc \ox' \rc)-1
= \dep(\frak x)+\dep(\frak x')-1 = k,\\
\dep(\ox)+\dep(\lc \ox'\rc) &=& \dep(\lc \ox \rc)+\dep(\lc \ox'\rc)-1
= \dep(\frak x)+ \dep(\frak x')-1 = k,\\
\dep(\ox)+ \dep(\ox') &=& \dep(\lc \ox \rc)-1+ \dep(\lc \ox' \rc)-1 =
\dep(\frak x)+\dep(\frak x')-2 = k-1.
\end{eqnarray*}

We next treat the case when $\bre(\frak x)>1$ or $\bre(\frak x')>1$. Let
$\frak x=\frak x_1\cdots\frak x_b$ and
$\frak x'=\frak x'_1\cdots\frak x'_{b'}$ be the standard decompositions
from Eq.~(\mref{eq:st}). We then define
\begin{equation}
\frak x \shpr \frak x'= \frak x_1\cdots \frak x_{b-1}(\frak x_b\shpr
\frak x'_1)\,
    \frak x'_{2}\cdots \frak x'_{b'}
\mlabel{eq:cdiam}
\end{equation}
where $\frak x_b\shpr \frak x'_1$ is defined by
Eq.~(\mref{eq:shprod}) and the rest is given by concatenation.
Extending $\shpr$ bilinearly, we obtain a binary operation
$$ \shpr: \FN(A)\otimes \FN(A) \to \FN(A).$$
This completes the definition of
$\shpr$.

\begin{lemma}
Let $\frak x,\frak x'\in \frak X_\infty$.
\begin{enumerate}
\item $h(\frak x)=h(\frak x\shpr \frak x')$ and $t(\frak x')=t(\frak x\shpr \frak x')$.
\mlabel{it:mat0}
\item If $t(\frak x)\neq h(\frak x')$, then
$\frak x \shpr \frak x' =\frak x \frak x'$ (concatenation).
\mlabel{it:mat1}

\item If $t(\frak x)\neq h(\frak x')$, then for any $\frak x''\in \frak X_\infty$,
$$(\frak x\frak x')\shpr \frak x'' =\frak x(\frak x' \shpr \frak x'')\,\text{ and }\,
\frak x''\shpr (\frak x \frak x') =(\frak x'' \shpr \frak x) \frak x'.$$
\mlabel{it:mat2}
\end{enumerate}
\mlabel{lem:match}
\end{lemma}

\begin{proof}
Items~(\mref{it:mat0}) and~(\mref{it:mat1}) follow from the definition of $\shpr$.
The proof of Item~(\mref{it:mat2}) is the same as~\cite[Lemma~4.4.5]{Gub}.
\end{proof}

We next define a linear operator
\begin{equation*}
\mop: \FN(A)  \to \FN(A) , \, \,\frakx \mapsto \lc \frak x \rc.
\end{equation*}
\emph{In the rest of the paper}, we will use the infix notation $\lc \frakx\rc$ interchangeably
with $\mop(\frakx)$ for any $\frakx\in \FN(A)$.
Let
$$j_X:X \hookrightarrow \frak X_\infty \hookrightarrow \FN(A)$$
be the natural injection which extends to an algebra injection
$$
j_A: A \to \FN(A).
$$

Now we state our first main result, to be proved in the
next subsection.
\begin{theorem}
Let $A$ be a  $\bf k$-algebra with a $\bf k$-basis $X$ and $\kappa\in \bfk$ be given.
\begin{enumerate}
\item
The pair $(\FN(A),\shpr)$ is an algebra.
\mlabel{it:alg}
\item
The triple $(\FN(A),\shpr,\mop)$ is a modified Rota-Baxter  algebra of weight $\kappa$.
\mlabel{it:MRB}
\item
The triple $(\FN(A),\shpr,\mop$) together with the embedding $j_A$ is the free modified Rota-Baxter  algebra of weight $\kappa$ on the algebra $A$. \mlabel{it:free}
\end{enumerate}
\mlabel{thm:freeao}
\end{theorem}

\subsection{The proof of Theorem~\mref{thm:freeao}}
\mlabel{ss:proof}

\begin{proof}
(\mref{it:alg}). It is enough to verify the associativity for basis elements:
\begin{equation}
 (\frak x'\shpr \frak x'')\shpr \frak x''' =\frak x'\shpr(\frak x'' \shpr \frak x'''),\,\text{ for all }  \frak x',\frak x'',\frak x'''\in \frak X_\infty.
\mlabel{eq:assx}
\end{equation}
We carry out the verification by induction on the sum of the depths
$$n:=\dep(\frak x')+\dep(\frak x'')+\dep(\frak x''')\geq 0.$$
If $n=0$, then $$\dep(\frak x') = \dep(\frak x'') = \dep(\frak x''') = 0$$
and so $\frak x',\frak x'',\frak x'''\in X$. In
this case the product $\shpr$ is given by the product in $A$ and so is associative.

Assume that Eq.~(\mref{eq:assx}) holds for $n\leq k$ for any given $k\geq 0$ and consider
$\frak x',\frak x'',\frak x'''\in \frak X_\infty$ with
$$n=\dep(\frak x')+\dep(\frak x'')+\dep(\frak x''')=k+1\geq 1.$$
If $t(\frak x')\neq h(\frak x'')$, then by Lemma~\mref{lem:match},
$$ (\frak x' \shpr \frak x'') \shpr \frak x'''=(\frak x'\frak x'')\shpr \frak x'''
= \frak x' (\frak x'' \shpr \frak x''') =\frak x'\shpr (\frak x''\shpr
\frak x''').$$
A similar argument holds when $t(\frak x'')\neq h(\frak x''')$.
Thus we only need to verify the associativity when
$t(\frak x')=h(\frak x'')$ and $t(\frak x'')=h(\frak x''')$. We next
reduce the proof to the breadths of the words and depart to show a lemma.

\begin{lemma}
If Eq.~(\mref{eq:assx})
holds for all $\frak x', \frak x''$ and $\frak x'''$ in $\frak X_\infty$
of breadth one, then it holds for all $\frak x', \frak x''$ and
$\frak x'''$ in $\frak X_\infty$. \mlabel{lem:ell}
\end{lemma}

\begin{proof}
We use induction on the sum of breadths
$m:=\bre(\frak x')+ \bre(\frak x'')+\bre(\frak x''')\geq 3$. The case
when $m=3$ is the assumption of the lemma. Assume the associativity
holds for $m\leq j$ for some $j\geq 3$ and
take $\frak x', \frak x'',\frak x'''\in
\frak X_\infty$ with $m = j+1\geq 4.$ So at least one of
$\frak x',\frak x'',\frak x'''$ has breadth greater than or equal to
2.

First assume that $\bre(\frak x')\geq 2$. Then we may write
$$\frak x'=\frak x'_1\frak x'_2,\,\text{ where } \frak x'_1,\, \frak x'_2\in \frak X_\infty\,\text{ and }\, t(\frak x'_1)\neq
h(\frak x'_2).$$
By Lemma~\mref{lem:match}, we obtain
$$
 (\frak x'\shpr \frak x'') \shpr \frak x'''=
((\frak x'_1\frak x'_2)\shpr \frak x'')\shpr \frak x'''
= (\frak x'_1 (\frak x'_2 \shpr \frak x''))\shpr \frak x'''
= \frak x'_1 ((\frak x'_2 \shpr \frak x'') \shpr \frak x'''). $$
Similarly,
$$ \frak x'\shpr (\frak x'' \shpr \frak x''')=
(\frak x'_1\frak x'_2)\shpr (\frak x''\shpr \frak x''')
= \frak x'_1 (\frak x'_2 \shpr (\frak x''\shpr \frak x''')).
$$
Thus
$$ (\frak x'\shpr \frak x'') \shpr \frak x'''=
  \frak x'\shpr (\frak x'' \shpr \frak x''')$$
whenever
$$ (\frak x'_2 \shpr \frak x'') \shpr \frak x'''=
\frak x'_2 \shpr (\frak x''\shpr \frak x'''),$$
which follows from the induction hypothesis.
A similar proof works if $\bre(\frak x''')\geq 2.$

Finally if $\bre(\frak x'')\geq 2$, we may write
$$\frak x''=\frak x''_1\frak x''_2\,\text{ where }\, \frak x''_1,\,\frak x''_2\in \frak X_\infty\,\text{ and }\, t(\frak x''_1)\neq h(\frak x''_2).$$
By Lemma~\mref{lem:match} again, we get
$$
(\frak x' \shpr \frak x'')\shpr \frak x'''=
(\frak x' \shpr (\frak x''_1 \frak x''_2)) \shpr \frak x''' \\
= ((\frak x' \shpr \frak x''_1)\frak x''_2)\shpr \frak x'''
= (\frak x'\shpr \frak x''_1)(\frak x''_2 \shpr \frak x''').
$$
In the same way, we have
$$(\frak x'\shpr \frak x''_1)(\frak x''_2 \shpr \frak x''')
= \frak x'\shpr (\frak x'' \shpr \frak x''').$$
This proves the associativity.

\end{proof}

In summary, the proof of the associativity has been reduced to the
special case when $\frak x',\frak x'',\frak x'''\in \frak X_\infty$ are
chosen so that
\begin{enumerate}
\item
$n= \dep(\frak x')+ \dep(\frak x'') + \dep(\frak x''')=k+1\geq 1$ with the
assumption that the associativity holds when $n\leq k$.
\mlabel{it:sp1}
\item
the elements have breadth one and \mlabel{it:sp2}
\item
$t(\frak x')=h(\frak x'')$ and $t(\frak x'')=h(\frak x''')$.
\mlabel{it:sp3}
\end{enumerate}
By Item~(\mref{it:sp2}), the head and tail of each of the elements are
the same. Therefore by Item~(\mref{it:sp3}), either all the three
elements are in $X$ or they are all in $\lc \frak X_\infty \rc$. If
all of $\frak x',\frak x'',\frak x'''$ are in $X$, then as already
shown, the associativity follows from the associativity in $A$.
So it remains to consider the case when $\frak x',\frak x'',\frak x'''$ are all in $\lc
\frak X_\infty \rc$. Then we may write
$$\frak x'=\lc \ox'\rc, \frak x''=\lc \ox''
\rc, \frak x'''=\lc \ox'''\rc \,\text{ with }\, \ox',\ox'',\ox'''\in
\frak X_\infty.$$
Applying Eq.~(\mref{eq:shprod}) and bilinearity of the
product $\shpr$, we get
\allowdisplaybreaks{
\begin{eqnarray*}
(\frak x'\shpr \frak x'')\shpr \frak x'''
&=& \big( \lc \lc \ox'\rc \shpr
\ox ''\rc +\lc\ox'\shpr \lc\ox''\rc\rc
    +\kappa\ox'\shpr \ox'' \big ) \shpr \lc \ox'''\rc \\
&=& \lc\lc \ox'\rc \shpr \ox''\rc \shpr \lc\ox'''\rc
    + \lc\ox'\shpr \lc \ox''\rc \rc\shpr \lc \ox'''\rc
    +\kappa(\ox'\shpr \ox'') \shpr \lc\ox'''\rc \\
&=&  \lc\lc \ox'\rc\shpr \ox''\rc\shpr \ox''' \rc
    + \lc\big(\lc\ox'\rc \shpr \ox''\big) \shpr \lc\ox'''\rc\rc
   +\kappa\big(\lc\ox'\rc \shpr\ox''\big)\shpr \ox''' \\
&& + \lc\lc\ox'\shpr\lc\ox''\rc\rc \shpr \ox'''\rc
    + \lc\big(\ox'\shpr\lc \ox''\rc\big) \shpr\lc \ox'''\rc\rc
    +\kappa\big(\ox'\shpr \lc \ox''\rc \big) \shpr \ox''' \\
&& +\kappa(\ox'\shpr \ox'') \shpr \lc\ox'''\rc \\
&=&  \lc\lc \ox'\rc\shpr \ox''\rc\shpr \ox''' \rc
    + \lc\big(\lc\ox'\rc \shpr \ox''\big) \shpr \lc\ox'''\rc\rc
    +\kappa\big(\lc\ox'\rc \shpr\ox''\big)\shpr \ox''' \\
&& + \lc\lc\ox'\shpr\lc\ox''\rc\rc \shpr \ox'''\rc\rc
    + \lc \ox'\shpr \lc \lc \ox''\rc \shpr\ox'''\rc\rc+ \lc \ox'\shpr \lc\ox''\shpr\lc\ox'''\rc\rc\\
&& +\kappa\lc \ox'\shpr \ox'' \shpr\ox'''\rc+\kappa\big(\ox'\shpr \lc \ox''\rc \big) \shpr \ox'''
 +\kappa(\ox'\shpr \ox'') \shpr \lc\ox'''\rc.
\end{eqnarray*}}
Similarly we obtain
\allowdisplaybreaks{
\begin{eqnarray*}
\frak x' \shpr \big(\frak x''\shpr \frak x'''\big)
&=&  \lc \ox'\rc \shpr \big(
 \lc \lc\ox ''\rc\shpr \ox'''\rc +\lc\ox''\shpr \lc\ox'''\rc\rc
    +\kappa\ox''\shpr \ox''' \big )  \\
&=&  \lc \ox'\rc \shpr
 \lc \lc\ox ''\rc\shpr \ox'''\rc +\lc \ox'\rc \shpr \lc\ox''\shpr \lc\ox'''\rc\rc
    +\kappa\lc \ox'\rc \shpr \big(\ox''\shpr \ox''' \big )  \\
&=& \lc\lc\ox'\rc\shpr \big(\lc\ox''\rc\shpr \ox'''\big )\rc
    + \lc \ox'\shpr \lc\lc\ox''\rc \shpr \ox'''\rc\rc
   +\kappa\ox'\shpr \big(\lc\ox''\rc\shpr\ox'''\big )\\
&&  + \lc\lc \ox'\rc\shpr \big(\ox''\shpr \lc \ox'''\rc \big) \rc
    + \lc \ox' \shpr \lc \ox'' \shpr \lc \ox'''\rc\rc\rc
    +\kappa \ox'\shpr \big( \ox''\shpr \lc \ox'''\rc \big)\\
&& +\kappa\lc \ox'\rc \shpr \big(\ox''\shpr \ox''' \big )\\
&=& \lc\lc\lc\ox'\rc\shpr \ox''\rc\shpr \ox'''\rc +\lc\lc\ox'\shpr \lc\ox''\rc\rc\shpr \ox'''\rc\\
&&   +\kappa\lc \ox'\shpr \ox'' \shpr\ox'''\rc + \lc \ox'\shpr \lc\lc\ox''\rc \shpr \ox'''\rc\rc
   +\kappa\ox'\shpr \big(\lc\ox''\rc\shpr\ox'''\big )\\
&&  + \lc\lc \ox'\rc\shpr \big(\ox''\shpr \lc \ox'''\rc \big) \rc
    + \lc \ox' \shpr \lc \ox'' \shpr \lc \ox'''\rc\rc\rc
    +\kappa \ox'\shpr \big( \ox''\shpr \lc \ox'''\rc \big)\\
&& +\kappa\lc \ox'\rc \shpr \big(\ox''\shpr \ox''' \big ).
\end{eqnarray*}}
Now by the induction hypothesis, the $i$-th term in the expansion of $(\frak x'\shpr
\frak x'')\shpr \frak x'''$ coincides with the $\sigma(i)$-th term  in
the expansion of $\frak x'\shpr(\frak x'' \shpr \frak x''')$. Here $\sigma\in \Sigma_{9}$ is the permutation given by
\begin{equation*}
\sigma= \left ( \begin{array}{ccccccccccc} 1&2&3&4&5&6&7&8&9\\
    1&6&9&2&4&7&3&5&8\end{array} \right ).
\end{equation*}
This completes the proof of Theorem~\mref{thm:freeao}~(\mref{it:alg}).

(\mref{it:MRB}). The proof follows from the definition
$\mop(\frak x)=\lc \frak x\rc$ and Eq. (\mref{eq:shprod}).

(\mref{it:free}). Let $(M,\ast,P)$ be a modified Rota-Baxter  algebra with multiplication $\ast$ and let $f:A\to M$ be a  $\bfk$-algebra homomorphism. We will construct a
$\bfk$-linear map $\free{f}:\FN(A) \to M$ by defining
$\free{f}(\frak x)$ for $\frak x\in \frak X_\infty$. We achieve this by
defining $\free{f}(\frak x)$ for $\frak x\in \frak X_n,\ n\geq 0$,
inductively on $n$. For $\frak x\in \frak X_0:=X$, define
$\free{f}(\frak x)=f(\frak x).$ Then $j \free{f}=f$ is satisfied. Suppose $\free{f}(\frak x)$ has been
defined for $\frak x\in \frak X_n$ and consider $\frak x$ in
$\frak X_{n+1}$ which is, by definition,
\begin{eqnarray*}
\altx(X,\frak X_{n}) =
    \Big( \bigsqcup_{r\geq 1} (X\lc \frak X_{n}\rc)^r \Big) \bigsqcup
    \Big(\bigsqcup_{r\geq 0} (X\lc \frak X_{n}\rc)^r  X\Big)  \bigsqcup \Big( \bigsqcup_{r\geq 0} \lc \frak X_{n}\rc (X\lc \frak X_{n}\rc)^r \Big)
   \bigsqcup \Big( \bigsqcup_{r\geq 0} \lc \frak X_{n}\rc (X\lc \frak X_{n}\rc)^r X\Big).
\end{eqnarray*}
Let $\frak x$ be in the first union component $\bigsqcup_{r\geq 1}
(X\lc \frak X_{n}\rc)^r$ above. Then
$$\frak x = \prod_{i=1}^r(\frak x_{2i-1} \lc \frak x_{2i} \rc)$$
for $\frak x_{2i-1}\in X$ and $\frak x_{2i}\in \frak X_n$, $1\leq i\leq
r$. By the construction of the multiplication $\shpr$ and the
modified Rota-Baxter  operator $\mop$, we have
$$\frak x= \shpr_{i=1}^r(\frak x_{2i-1} \shpr \lc \frak x_{2i}\rc)
    = \shpr_{i=1}^r(\frak x_{2i-1} \shpr \mop(\frak x_{2i})).$$
Define
\begin{equation}
\free{f}(\frak x) = \ast_{i=1}^r \big(\free{f}(\frak x_{2i-1})
    \ast \mop\big (\free{f}(\frak x_{2i})) \big).
\mlabel{eq:hom}
\end{equation}
where the right hand side is well-defined by the induction
hypothesis. Similarly define $\free{f}(\frak x)$ if $\frak x$ is in
the other union components. For any $\frak x\in \frak X_\infty$, we
have $\mop(\frak x)=\lc \frak x\rc\in \frak X_\infty$, and by the definition of $\free{f}$ in (Eq. (\mref{eq:hom})), we have
\begin{equation}
\free{f}(\lc \frak x \rc)=P(\free{f}(\frak x)). \mlabel{eq:hom1-2}
\end{equation}
So $\free{f}$ commutes with the modified Rota-Baxter  operators. Combining this
equation with Eq.~(\mref{eq:hom}) we see that if
$\frak x=\frak x_1\cdots \frak x_b$ is the standard decomposition of
$\frak x$, then
\begin{equation*}
 \free{f}(\frak x)=\free{f}(\frak x_1)*\cdots * \free{f}(\frak x_b).
\end{equation*}

Note that this is the only possible way to define $\free{f}(\frak x)$
in order for $\free{f}$ to be a modified Rota-Baxter  algebra homomorphism extending $f$.
It remains to prove that the map $\free{f}$ defined in
Eq.~(\mref{eq:hom}) is indeed an algebra homomorphism. For this we
only need to check the multiplicity
\begin{equation}
\free{f} (\frak x \shpr \frak x')=\free{f}(\frak x) \ast
\free{f}(\frak x') \mlabel{eq:hom2}
\end{equation}
for all $\frak x,\frak x'\in \frak X_\infty$. For this we use induction
on the sum of depths $n:=\bre(\frak x)+\bre(\frak x')$. Then $n\geq 0$. When
$n=0$, we have $\frak x,\frak x'\in X$. Then Eq.~(\mref{eq:hom2})
follows from the multiplicity of $f$. Assume the multiplicity holds
for $\frak x,\frak x' \in \frak X_\infty$ with $n\geq k$ and take
$\frak x,\frak x'\in \frak X_\infty$ with $n=k+1$. Let
$\frak x=\frak x_1\cdots \frak x_b$ and
$\frak x'=\frak x'_1\cdots\frak x'_{b'}$ be the standard
decompositions. Since $n=k+1\geq 1$, at least one of $\frak x_b$ and $\frak x'_{b'}$ is in $\lc \frak X_\infty\rc$. Then by Eq.~(\mref{eq:shprod}) we have

\begin{align*}
\free{f}(\frak x_b\shpr \frak x'_1)&= \left \{\begin{array}{ll}
\free{f}(\frak x_b \frak x'_1), & {\rm if\ } \frak x_b\in X, \frak x'_1\in \lc \frak X_\infty\rc,\\
\free{f}(\frak x_b \frak x'_1), & {\rm if\ } \frak x_b\in \lc
\frak X_\infty\rc,
    \frak x'_1\in X,\\
\free{f}\big( \lc \lc \ox_b\rc \shpr \ox'_1\rc +\lc \ox_b \shpr \lc
\ox'_1\rc \rc+\kappa \ox_b \shpr \ox'_1 \big), & {\rm if\ }
\frak x_b=\lc \ox_b\rc, \frak x'_1=\lc \ox'_1\rc \in \lc \frak X_\infty
\rc.
\end{array} \right .
\end{align*}
In the first two cases, the right hand side is
$\free{f}(\frak x_b)*\free{f}(\frak x'_1)$ by the definition of
$\free{f}$. In the third case, applying Eq.~(\mref{eq:hom1-2}),
the induction hypothesis and the modified Rota-Baxter  relation of the operator $P$ on $M$, we have
\begin{eqnarray*}
&&\free{f}\big( \lc \lc \ox_b\rc \shpr \ox'_1\rc
    +\lc \ox_b \shpr \lc \ox'_1\rc \rc
+\kappa \ox_b \shpr \ox'_1 \big)\\
&=&\free{f}(\lc \lc \ox_b\rc \shpr \ox'_1\rc) + \free{f}(\lc \ox_b
\shpr \lc \ox'_1\rc \rc)
+\kappa\free{f}( \ox_b \shpr \ox'_1 )\\
&=&P(\free{f}(\lc \ox_b\rc \shpr \ox'_1)) + P(\free{f}(\ox_b \shpr
\lc \ox'_1\rc ))
+\kappa\free{f}(\ox_b \shpr \ox'_1 )\\
&=&P(\free{f}(\lc \ox_b\rc)*\free{f}(\ox'_1)) + P(\free{f}(\ox_b)
*\free{f}( \lc \ox'_1\rc ))
+\kappa\free{f}(\ox_b) * \free{f}(\ox'_1)\\
&=&P(P(\free{f}(\ox_b))*\free{f}(\ox'_1)) + P(\free{f}(\ox_b)
*P(\free{f}(\ox'_1)))
+\kappa(\free{f}(\ox_b) * \free{f}(\ox'_1))\\
&=& P(\free{f}(\ox_b))*P(\free{f}(\ox'_1))\\
&=& \free{f}(\lc \ox_b\rc) * \free{f}(\lc\ox'_1\rc)\\
&=& \free{f} (\frak x_b) *\free{f}(\frak x'_1).
\end{eqnarray*}
Therefore $\free{f}(\frak x_b\shpr
\frak x'_1)=\free{f}(\frak x_b)*\free{f}(\frak x'_1)$. Then
\begin{eqnarray*}
\free{f}(\frak x\shpr \frak x')&=&
\free{f}\big(\frak x_1\cdots\frak x_{b-1}(\frak x_b\shpr
\frak x'_1)\frak x'_2\cdots
    \frak x'_{b'}\big) \\
&=& \free{f}(\frak x_1)*\cdots *\free{f}(\frak x_{b-1})*
\free{f}(\frak x_b\shpr \frak x'_1)*\free{f}(\frak x'_2)\cdots
    \free{f}(\frak x'_{b'})\\
&=& \free{f}(\frak x_1)*\cdots *\free{f}(\frak x_{b-1})*
\free{f}(\frak x_b)* \free{f} (\frak x'_1)*\free{f}(\frak x'_2)\cdots
    \free{f}(\frak x'_{b'})\\
&=& \free{f}(\frak x)*\free{f}(\frak x'),
\end{eqnarray*}
as required.

This completes the proof of Theorem~\mref{thm:freeao}
\end{proof}

\section{The Hopf algebra structure on free modified Rota-Baxter algebras}
\mlabel{sec:hopf}

In this section, starting with the assumption that $A$ is a bialgebra with its coproduct $\Delta_A$ and its counit $\vep_A$, we provide a bialgebraic and then a Hopf algebraic structure on the free modified Rota-Baxter algebras $F_\kappa(A)$ obtained in Section~\mref{sec:free}, when $\kappa=-\lambda^2$. It would be interesting to see how to extend this construction to other weights $\kappa$. For Hopf algebra structures on free Rota-Baxter algebras, see~\cite{GGZ,ZtGG} for Hopf algebra structures on free Rota-Baxter algebras.

\subsection{The bialgebraic structure}
\mlabel{ss:coass}
We now build on results from previous subsections to obtain a bialgebra structure on $\FNN(A)$.
We first record some lemmas for a preparation.

\begin{lemma}
Let $\lambda$ be a given element of $\bfk$.
\begin{enumerate}
\item The linear map $-\lambda\id :\bfk \rightarrow \bfk$ is a modified Rota-Baxter operator of weight $-\lambda^2$ on $\bfk$. \mlabel{it:mrbl1}

\item  There exists a unique modified Rota-Baxter algebra morphism $\vep_\mo:\FNN(A)\rightarrow \bfk$ such that
\begin{align}
\vep_\mo \circ j_A =\vep_A\,\text{ and }\, \vep_\mo\circ P_A=-\lambda\id \circ \vep_\mo.\mlabel{eq:countmor}
\end{align}  \mlabel{it:mrbl2}
\end{enumerate}
\mlabel{lem:mrbl}
\end{lemma}

\begin{proof}
(\mref{it:mrbl1}) It follows from
\begin{align*}
(-\lambda\id)(a) (-\lambda\id)(b) =&\ \lambda^2 a b= \lambda^2 a b + \lambda^2 a b - \lambda^2 a b \\
=&\ (-\lambda\id)( a(-\lambda\id)(b)) +  (-\lambda\id)( (-\lambda\id)(a)b) - \lambda^2 ab.
\end{align*}

(\mref{it:mrbl2}) By Item~(\mref{it:mrbl1}), $(\bfk, -\lambda\id)$ is a modified Rota-Baxter algebra of weight $-\lambda^2$.
Then the remainder follows from Theorem~\mref{thm:freeao}~(\mref{it:free}).
\end{proof}

Note that $P_A$ is a modified Rota-Baxter operator on $\FNN(A)$;however $P_A \ot P_A$ is not a modified Rota-Baxter operator on $\FNN(A) \ot \FNN(A)$.
The following result constructs a modified Rota-Baxter operator on $\FNN(A) \ot \FNN(A)$.

\begin{lemma}\mlabel{lem:RBO1}
Let $\lambda$ be a given element of $\bfk$. Define the linear map
$$ \otp:\FNN(A)\ot \FNN(A)\rightarrow \FNN(A)\ot \FNN(A)$$ by taking
\begin{align}
\otp(\frak x \ot \frak x') :=(P_A(\frak x)+\lambda \frak x)\ot \vep_\mo(\frak x')1+\frak x\ot P_A(\frak x')\, \text{ for }\, \frak x, \frak x'\in \FNN(A).
\mlabel{eq:mRB}
\end{align}
Then $\otp$ is a modified Rota-Baxter operator of weight $-\lambda^2$ on $\FNN(A) \ot \FNN(A)$.
\end{lemma}

\begin{proof}
Let $\frak x_1, \frak x_2, \frak x'_1, \frak x'_2\in \FNN(A)$. On the one hand,
\begin{align*}
&\ \otp(\frak x_1 \ot \frak x'_1)\diamond' \otp(\frak x_2 \ot \frak x'_2)\\
=&\ \Big((P_A(\frak x_1)+\lambda \frak x_1)\ot \vep_\mo(\frak x'_1)1+\frak x_1\ot P_A(\frak x'_1)\Big)\diamond' \Big((P_A(\frak x_2)+\lambda \frak x_2)\ot \vep_\mo(\frak x'_2)1+\frak x_2\ot P_A(\frak x'_2)\Big)\\
=&\ \Big(\big(P_A(\frak x_1)+\lambda \frak x_1\big)\diamond \big(P_A(\frak x_2)+\lambda \frak x_2\big)\Big)\ot \vep_\mo(\frak x'_1)\vep_\mo(\frak x'_2)1+\Big(\big(P_A(\frak x_1)+\lambda \frak x_1\big)\diamond \frak x_2\Big)\ot \vep_\mo(\frak x'_1)P_A(\frak x'_2)\\
&\ + \Big(\frak x_1 \diamond\big(P_A(\frak x_2)+\lambda \frak x_2\big)\Big)\ot P_A(\frak x'_1)\vep_\mo(\frak x'_2)+(\frak x_1 \diamond \frak x_2)\ot \Big(P_A(\frak x'_1)\diamond P_A(\frak x'_2)\Big)\\
=&\ \Big(P_A(\frak x_1) \diamond P_A(\frak x_2)+\lambda P_A(\frak x_1)\diamond \frak x_2+
 \lambda \frak x_1 \diamond  P_A(\frak x_2)+\lambda^2 \frak x_1\diamond \frak x_2\Big)\ot \vep_\mo(\frak x'_1)\vep_\mo(\frak x'_2)1\\
&\ +\Big(P_A(\frak x_1)\diamond \frak x_2\Big) \ot \vep_\mo(\frak x'_1)P_A(\frak x'_2)+\lambda( \frak x_1\diamond \frak x_2) \ot \vep_\mo(\frak x'_1)P_A(\frak x'_2)+ \Big(\frak x_1 \diamond P_A(\frak x_2) \Big)\ot P_A(\frak x'_1)\vep_\mo(\frak x'_2)\\
&\ +\lambda \big(\frak x_1 \diamond \frak x_2  \big)\ot P_A(\frak x'_1)\vep_\mo(\frak x'_2)+(\frak x_1 \diamond \frak x_2) \ot \Big(P_A(\frak x'_1)\diamond P_A(\frak x'_2)\Big) \\
=&\ \bigg(P_A(\frak x_1 \diamond P_A(\frak x_2)) + P_A( P_A(\frak x_1)\diamond \frak x_2)+\lambda P_A(\frak x_1)\diamond \frak x_2+
 \lambda \frak x_1 \diamond  P_A(\frak x_2)\bigg)\ot \vep_\mo(\frak x'_1)\vep_\mo(\frak x'_2)1\\
&\ +\big(P_A(\frak x_1)\diamond \frak x_2\big) \ot \vep_\mo(\frak x'_1)P_A(\frak x'_2)+\lambda( \frak x_1\diamond \frak x_2) \ot \vep_\mo(\frak x'_1)P_A(\frak x'_2)+ \big(\frak x_1 \diamond P_A(\frak x_2) \big)\ot P_A(\frak x'_1)\vep_\mo(\frak x'_2)\\
&\ +\lambda \big(\frak x_1 \diamond \frak x_2  \big)\ot P_A(\frak x'_1)\vep_\mo(\frak x'_2)+(\frak x_1 \diamond \frak x_2) \ot \bigg(P_A(\frak x_1 \diamond P_A(\frak x_2)) + P_A( P_A(\frak x_1)\diamond \frak x_2)-\lambda^2 \frak x'_1\diamond \frak x'_2\bigg).\\
& \hspace{9cm}  \text{(by Theorem~(\ref{thm:freeao})~(\mref{it:MRB}))}
\end{align*}
On the other hand,
\begin{align*}
&\ \otp\Big((\frak x_1 \ot \frak x'_1)\diamond' \otp(\frak x_2 \ot \frak x'_2)\Big)+\otp\Big(\otp(\frak x_1 \ot \frak x'_1)\diamond' (\frak x_2 \ot \frak x'_2)\Big)-\lambda^2\Big((\frak x_1 \ot \frak x'_1)\diamond' (\frak x_2 \ot \frak x'_2)\Big)\\
=&\ \otp\bigg((\frak x_1 \ot \frak x'_1) \diamond' \Big((P_A(\frak x_2)+\lambda \frak x_2)\ot \vep_\mo(\frak x'_2)1+\frak x_2 \ot P_A(\frak x'_2) \Big)\bigg)\\
&\ +\otp\bigg( \Big((P_A(\frak x_1)+\lambda \frak x_1 )\ot \vep_\mo(\frak x'_1 )1+\frak x_1 \ot P_A(\frak x'_1) \Big) \diamond'  (\frak x_2 \ot \frak x'_2)  \bigg)- \lambda^2 (\frak x_1 \diamond \frak x_2) \ot (\frak x'_1 \diamond \frak x'_2)\\
=&\ \otp\bigg( \big(\frak x_1 \diamond P_A(\frak x_2)\big) \ot \frak x'_1  \vep_\mo (\frak x'_2) +\lambda (\frak x_1 \diamond \frak x_2)\ot \frak x'_1  \vep_\mo (\frak x'_2)+(\frak x_1 \diamond \frak x_2)\ot \big(\frak x'_1 \diamond P_A(\frak x'_2)\big) \bigg)\\
&\ + \otp\bigg(  \big(P_A(\frak x_1) \diamond \frak x_2\big) \ot \vep_\mo (\frak x'_1)\frak x'_2 +\lambda (\frak x_1 \diamond \frak x_2)\ot \vep_\mo (\frak x'_1) \frak x'_2+(\frak x_1 \diamond \frak x_2)\ot \big(P_A(\frak x'_1) \diamond \frak x'_2\big) \bigg)\\
&\ - \lambda^2 (\frak x_1 \diamond \frak x_2) \ot (\frak x'_1 \diamond \frak x'_2)\\
=&\ \bigg(P_A\big(\frak x_1 \diamond P_A(\frak x_2)\big)+\lambda\frak x_1 \diamond P_A(\frak x_2) \bigg) \ot \vep_\mo \big(\frak x'_1 \vep_\mo(\frak x'_2)\big)1+\big(\frak x_1 \diamond P_A(\frak x_2)\big) \ot P_A\big(\frak x'_1 \vep_\mo(\frak x'_2)\big)\\
&\ +\lambda \bigg( P_A(\frak x_1 \diamond \frak x_2 ) +\lambda \frak x_1 \diamond \frak x_2 \bigg)\ot \vep_\mo \big(\frak x'_1 \vep_\mo(\frak x'_2)\big)1+\lambda (\frak x_1 \diamond \frak x_2 ) \ot P_A\big(\frak x'_1 \vep_\mo(\frak x'_2)\big)\\
&\ + \bigg(P_A(\frak x_1 \diamond \frak x_2 )+\lambda \frak x_1 \diamond \frak x_2 \bigg)\ot \vep_\mo \big(\frak x'_1 P_A(\frak x'_2)\big)1+(\frak x_1 \diamond \frak x_2) \ot P_A \big(\frak x'_1 P_A(\frak x'_2)\big)\\
&\ + \bigg(P_A\big(P_A(\frak x_1) \diamond\frak x_2\big)+\lambda P_A(\frak x_1) \diamond\frak x_2 \bigg) \ot \vep_\mo \big(\vep_\mo(\frak x'_1)\frak x'_2\big)1+\big(P_A(\frak x_1) \diamond\frak x_2\big) \ot P_A\big(\vep_\mo(\frak x'_1)\frak x'_2\big)\\
&\ +\lambda \bigg( P_A(\frak x_1 \diamond \frak x_2 ) +\lambda \frak x_1 \diamond \frak x_2 \bigg)\ot \vep_\mo \big(\vep_\mo(\frak x'_1)\frak x'_2\big)1+\lambda (\frak x_1 \diamond \frak x_2 ) \ot P_A\big(\vep_\mo(\frak x'_1 )\frak x'_2\big)\\
&\ + \bigg(P_A(\frak x_1 \diamond \frak x_2 )+\lambda \frak x_1 \diamond \frak x_2 \bigg)\ot \vep_\mo \big(P_A(\frak x'_1)\frak x'_2\big)+(\frak x_1 \diamond \frak x_2) \ot P_A \big(P_A(\frak x'_1)\frak x'_2\big)- \lambda^2 (\frak x_1 \diamond \frak x_2) \ot (\frak x'_1 \diamond \frak x'_2)\\
=&\ \bigg(P_A\big(\frak x_1 \diamond P_A(\frak x_2)\big)+\lambda\frak x_1 \diamond P_A(\frak x_2) \bigg) \ot \vep_\mo (\frak x'_1) \vep_\mo(\frak x'_2)1+\big(\frak x_1 \diamond P_A(\frak x_2)\big) \ot P_A(\frak x'_1) \vep_\mo(\frak x'_2)\\
&\ +\lambda \bigg( P_A(\frak x_1 \diamond \frak x_2 ) +\lambda \frak x_1 \diamond \frak x_2 \bigg)\ot \vep_\mo (\frak x'_1)  \vep_\mo(\frak x'_2)1+\lambda (\frak x_1 \diamond \frak x_2 ) \ot P_A(\frak x'_1) \vep_\mo(\frak x'_2)\\
&\ + \bigg(P_A(\frak x_1 \diamond \frak x_2 )+\lambda \frak x_1 \diamond \frak x_2 \bigg)\ot \vep_\mo (\frak x'_1) \vep_\mo (P_A(\frak x'_2))1+(\frak x_1 \diamond \frak x_2) \ot P_A \big(\frak x'_1 P_A(\frak x'_2)\big)\\
&\ + \bigg(P_A\big(P_A(\frak x_1) \diamond\frak x_2\big)+\lambda P_A(\frak x_1) \diamond\frak x_2 \bigg) \ot \vep_\mo(\frak x'_1) \vep_\mo (\frak x'_2)1+\big(P_A(\frak x_1) \diamond\frak x_2\big) \ot \vep_\mo(\frak x'_1)P_A\big(\frak x'_2\big)\\
&\ +\lambda \bigg( P_A(\frak x_1 \diamond \frak x_2 ) +\lambda \frak x_1 \diamond \frak x_2 \bigg)\ot \vep_\mo(\frak x'_1)\vep_\mo(\frak x'_2)1+\lambda (\frak x_1 \diamond \frak x_2 ) \ot \vep_\mo(\frak x'_1 )P_A\big(\frak x'_2\big)\\
&\ + \bigg(P_A(\frak x_1 \diamond \frak x_2 )+\lambda \frak x_1 \diamond \frak x_2 \bigg)\ot \vep_\mo (P_A(\frak x'_1)) \vep_\mo(\frak x'_2)1+(\frak x_1 \diamond \frak x_2) \ot P_A \big(P_A(\frak x'_1)\frak x'_2\big)\\
&- \lambda^2 (\frak x_1 \diamond \frak x_2) \ot (\frak x'_1 \diamond \frak x'_2)
\hspace{1cm}  \text{(by $\vep_\mo$ is $\bfk$-linear and $\vep_\mo$ is a homomorphism)}\\
=&\ \bigg(P_A\big(\frak x_1 \diamond P_A(\frak x_2)\big)+\lambda\frak x_1 \diamond P_A(\frak x_2) \bigg) \ot \vep_\mo (\frak x'_1) \vep_\mo(\frak x'_2)1+\big(\frak x_1 \diamond P_A(\frak x_2)\big) \ot P_A(\frak x'_1) \vep_\mo(\frak x'_2)\\
&\ +\lambda \bigg( P_A(\frak x_1 \diamond \frak x_2 ) +\lambda \frak x_1 \diamond \frak x_2 \bigg)\ot \vep_\mo (\frak x'_1)  \vep_\mo(\frak x'_2)1+\lambda (\frak x_1 \diamond \frak x_2 ) \ot P_A(\frak x'_1) \vep_\mo(\frak x'_2)\\
&\ -\lambda \bigg(P_A(\frak x_1 \diamond \frak x_2 )+\lambda \frak x_1 \diamond \frak x_2 \bigg)\ot \vep_\mo (\frak x'_1) \vep_\mo (\frak x'_2)1+(\frak x_1 \diamond \frak x_2) \ot P_A \big(\frak x'_1 P_A(\frak x'_2)\big)\\
&\ + \bigg(P_A\big(P_A(\frak x_1) \diamond\frak x_2\big)+\lambda P_A(\frak x_1) \diamond\frak x_2 \bigg) \ot \vep_\mo(\frak x'_1) \vep_\mo (\frak x'_2)1+\big(P_A(\frak x_1) \diamond\frak x_2\big) \ot \vep_\mo(\frak x'_1)P_A\big(\frak x'_2\big)\\
&\ +\lambda \bigg( P_A(\frak x_1 \diamond \frak x_2 ) +\lambda \frak x_1 \diamond \frak x_2 \bigg)\ot \vep_\mo(\frak x'_1)\vep_\mo(\frak x'_2)1+\lambda (\frak x_1 \diamond \frak x_2 ) \ot \vep_\mo(\frak x'_1 )P_A\big(\frak x'_2\big)\\
&\ -\lambda \bigg(P_A(\frak x_1 \diamond \frak x_2 )+\lambda \frak x_1 \diamond \frak x_2 \bigg)\ot \vep_\mo (\frak x'_1) \vep_\mo(\frak x'_2)1+(\frak x_1 \diamond \frak x_2) \ot P_A \big(P_A(\frak x'_1)\frak x'_2\big) \\
&- \lambda^2 (\frak x_1 \diamond \frak x_2) \ot (\frak x'_1 \diamond \frak x'_2)
\hspace{1cm}  \text{(Using Eq.~(\mref{eq:countmor}) in the fifth and eleventh terms)}\\
=&\ \bigg(P_A\big(\frak x_1 \diamond P_A(\frak x_2)\big)+\lambda\frak x_1 \diamond P_A(\frak x_2)+P_A\big(P_A(\frak x_1) \diamond\frak x_2\big)+\lambda P_A(\frak x_1) \diamond\frak x_2  \bigg) \ot \vep_\mo (\frak x'_1) \vep_\mo(\frak x'_2)1\\
&\ +\big(\frak x_1 \diamond P_A(\frak x_2)\big) \ot P_A(\frak x'_1) \vep_\mo(\frak x'_2)
 +\lambda (\frak x_1 \diamond \frak x_2 ) \ot P_A(\frak x'_1) \vep_\mo(\frak x'_2)
 +(\frak x_1 \diamond \frak x_2) \ot P_A \big(\frak x'_1 P_A(\frak x'_2)\big)\\
&\ +\big(P_A(\frak x_1) \diamond\frak x_2\big) \ot \vep_\mo(\frak x'_1)P_A\big(\frak x'_2\big)
 +\lambda (\frak x_1 \diamond \frak x_2 ) \ot \vep_\mo(\frak x'_1 )P_A\big(\frak x'_2\big)
 +(\frak x_1 \diamond \frak x_2) \ot P_A \big(P_A(\frak x'_1)\frak x'_2\big)\\
&\ - \lambda^2 (\frak x_1 \diamond \frak x_2) \ot (\frak x'_1 \diamond \frak x'_2).
\end{align*}
This completes the proof.
\end{proof}

With a similar argument, we can obtain

\begin{lemma}\mlabel{lem:RBO2}
Let $\lambda$ be a given element of $\bfk$. Define the linear map
$$\ottp:\FNN(A)\ot \FNN(A) \ot \FNN(A) \rightarrow \FNN(A)\ot \FNN(A) \ot \FNN(A)$$ by taking
\begin{equation}
\begin{aligned}
\ottp(\frak x \ot \frak x' \ot \frak x'') :=&(P_A(\frak x)+\lambda \frak x)\ot \vep_\mo(\frak x')1\ot \vep_\mo(\frak x'')1+\frak x\ot (P_A(\frak x')+\lambda \frak x') \ot \vep_\mo(\frak x'')1 \\
&+\frak x \ot \frak x' \ot P_A(\frak x'')\,
\text{ for }\, \frak x, \frak x'\in \FNN(A).
\mlabel{eq:mRB1}
\end{aligned}
\end{equation}
Then $\ottp$ is a modified Rota-Baxter operator of weight $-\lambda^2$ on $\FNN(A)\ot \FNN(A) \ot \FNN(A)$.
\end{lemma}

Now we are ready for our main result of this subsection.
Recall  $\vep_\mo: \FNN(A) \to \bfk$ is an algebra homomorphism given in Lemma~\mref{lem:mrbl}.
Let $j_A: A \to \FN(A)$ be the natural embedding. By Theorem~\mref{thm:freeao}~(\mref{it:free}) and Lemma~\mref{lem:RBO1}, there is a (unique) modified Rota-Baxter algebra  morphism
$$\Delta_M: \FNN(A) \rightarrow \FNN(A)\ot \FNN(A)$$ such that $\Delta_M \circ j_A=\Delta_A$.

\begin{theorem}
Let $A$ be a bialgebra and $\lambda\in \bfk$. Then the quintuple $(\FNN(A), \diamond, 1, \Delta_\mo, \vep_\mo)$ is a bialgebra.
\mlabel{thm:bialg}
\end{theorem}

\begin{proof}
It suffices to prove the counity of $\vep_\mo$ and coassociativity of $\Delta_\mo$.
For the former, denote by
$$\phi := (\vep_\mo \ot \id ) \Delta_M: \FNN(A) \to \FNN(A)$$
Then $\phi$ is an algebra homomorphism, since $\vep_\mo$ and $\Delta_M$ are algebra homomorphisms.
Further it is a modified Rota-Baxter algebra  morphism. Indeed,
for any $\frak x\in \FNN(A)$,
\begin{align*}
\phi \circ P_A(\frak x)=&\big((\vep_\mo \ot \id ) \Delta_M \big)P_A(\frak x)=(\vep_\mo \ot \id ) (\Delta_M P_A)(\frak x)\\
=&(\vep_\mo \ot \id )  (\otp\Delta_M)(\frak x)\quad(\text{by $\Delta_M$ being an modified Rota-Baxter algebra morphism}) \\
=&(\vep_\mo \ot \id ) \otp\Big(\sum_{(\frak x)}\frak x_{(1)} \ot \frak x_{(2)}\Big) \quad(\text{by Sweedler's notation})\\
=&\sum_{(\frak x)}(\vep_\mo \ot \id )\Big((P_A(\frak x_{(1)})+\lambda \frak x_{(1)})\ot \vep_\mo(\frak x_{(2)})1+\frak x_{(1)}\ot P_A(\frak x_{(2)}) \Big)\quad(\text{by Eq.~(\mref{eq:mRB})})\\
=&\sum_{(\frak x)}\bigg(\vep_\mo (P_A(\frak x_{(1)})) \ot \vep_\mo(\frak x_{(2)})1 +\lambda \vep_\mo (\frak x_{(1)})\ot \vep_\mo(\frak x_{(2)})1+\vep_\mo (\frak x_{(1)})\ot P_A(\frak x_{(2)}) \bigg)\\
=&\sum_{(\frak x)}\bigg(-\lambda\vep_\mo (\frak x_{(1)}) \ot \vep_\mo(\frak x_{(2)})1 +\lambda \vep_\mo (\frak x_{(1)})\ot \vep_\mo(\frak x_{(2)})1+\vep_\mo (\frak x_{(1)})\ot P_A(\frak x_{(2)}) \bigg)\\
& \hspace{7cm} (\text{by Eq.~(\mref{eq:countmor})})\\
=&\sum_{(\frak x)} \vep_\mo(\frak x_{(1)})\ot P_A(\frak x_{(2)}) = \sum_{(\frak x)} P_A(\vep_\mo(\frak x_{(1)})\frak x_{(2)})
= P_A((\vep_\mo \ot \id ) \Delta_M(\frakx))\\
=&P_A\circ \phi(\frakx).
\end{align*}
By unicity in the universal property of $\FNN(A)$, we have
$$(\vep_\mo \ot \id ) \Delta_M = \phi =\id _{\FNN(A)}$$ and so $\vep_\mo$ is a left counit.
By symmetry, we can prove $\vep_\mo$ is also a right counit.

Moreover, both $(\Delta_M\ot \id )\Delta_M$ and $(\id \ot \Delta_M )\Delta_M$ are modified Rota-Baxter algebra morphisms from $\FNN(A)$ to $\FNN(A)\ot \FNN(A) \ot \FNN(A)$, which is equipped with the modified Rota-Baxter operator $\ottp$ of weight $-\lambda^2$ given in Lemma~\mref{lem:RBO2}.
As they coincide on $A$
\[
(\Delta_M\ot \id )\Delta_M |_A = (\Delta_A\ot \id )\Delta_A = (\id \ot \Delta_A)\Delta_A = (\id \ot \Delta_M)\Delta_M |_A,
\]
they are equal and so $\Delta_M$ is coassociative. Here $\Delta_A$ is the coproduct on $A$.
Thus the quintuple $(\FNN(A), \diamond, 1, \Delta_\mo, \vep_\mo)$ is a bialgebra.
\end{proof}

\begin{remark}
For any $\frak x\in \FNN(A)$, we have
\begin{align*}
\Delta_M \circ P_A(\frak x)=&\otp\circ \Delta_M(\frak x)\quad(\text{by $\Delta_M$ being a modified Rota-Baxter algebra morphism}) \\
=&\otp\Big(\sum_{(\frak x)}\frak x_{(1)} \ot \frak x_{(2)}\Big) \quad(\text{by Sweedler's notation})\\
=&\sum_{(\frak x)}\Big((P_A(\frak x_{(1)})+\lambda \frak x_{(1)})\ot \vep_\mo(\frak x_{(2)})1+\frak x_{(1)}\ot P_A(\frak x_{(2)}) \Big)\quad(\text{by Eq.~(\mref{eq:mRB})})\\
=&\sum_{(\frak x)}\Big(P_A(\frak x_{(1)})\ot \vep_\mo(\frak x_{(2)})1 +\lambda \frak x_{(1)}\ot \vep_\mo(\frak x_{(2)})1+\frak x_{(1)}\ot P_A(\frak x_{(2)}) \Big)\\
=&\sum_{(\frak x)}\Big(P_A(\frak x_{(1)}\vep_\mo(\frak x_{(2)}))\ot1 +\lambda \frak x_{(1)}\vep_\mo(\frak x_{(2)})\ot1+(\id \ot P_A)(\frak x_{(1)}\ot \frak x_{(2)} )\Big)\\
=&P_{A}(\frak x)\ot 1+\lambda \frak x\ot 1+(\id \ot P_A)\Delta_M(\frak x) \quad (\text{by $\vep_M$ being the counit}).
\end{align*}
In other words,
\begin{align}
\Delta_M  P_A=P_A \ot 1 +(\id \ot P_A)\Delta_M+ \lambda\id \ot 1,\mlabel{eq:cocyc}
\end{align}
which is analogue to the 1-cocycle condition in the well-known Connes-Kreimer Hopf algebra on rooted trees~\mcite{CK}.
\end{remark}

\subsection{The Hopf algebraic structure}
\mlabel{ss:hopf}

In this last part of the paper we show that if we start with $A$ being a connected filtered bialgebra and $\lambda\in \bfk$, then the bialgebra $\FNN(A)$ also has a connected filtration and hence is a Hopf algebra.

\begin{defn}
A bialgebra $(A, m, \mu, \Delta, \vep)$ is called {\bf filtered} if it has an increasing filtration $A_n$, $n\geq 0$, such that
$$A=\cup_{n\geq 0} A_n,\, A_p A_q\subseteq A_{p+q} \,\text{ and }\,  \Delta(A_n)\subseteq \sum_{p+q=n}A_p\ot A_q \,\text{ for }\,  p, q, n\geq 0.$$
A filtered bialgebra $A$ is called {\bf connected} if  $A_0 = \im\, \mu $ and $A = A_0 \oplus \ker \varepsilon$.
\end{defn}
The following result is well-known.

\begin{lemma}{\rm \cite{DM}}
A connected filtered bialgebra is a Hopf algebra.
\mlabel{lem:hopf}
\end{lemma}

Our discussion in this section will be based on the following condition.
\begin{defn}
A $\bfk$-basis $X$ of a connected filtered bialgebra $A=\cup_{n\geq 0}A_n$ is called a {\bf filtered basis} of $A$ if there is an increasing filtration $X=\cup_{n\geq 0} X_n$ such that
$$A_n=\bfk X_n, X\backslash \{1\} \subseteq \ker \vep, X_0=\{1\}.$$
Here $1$ is the identity of $A$. Elements  $x\in X_n\setminus X_{n-1}$ are said to have {\bf degree} $n$, denoted by $\deg_A(x)=n$.
\mlabel{de:filbase}
\end{defn}

Let $A$ be a connected filtered bialgebra with a filtered basis $X$.
Recall that $\frakX_\infty$ constructed in Subsection~\mref{ss:bmrb} is a \bfk-basis of the free modified Rota-Baxter algebra $\FNN(A)$.
We now define the {\bf degree} $\deg(\frakx)$ for $\frakx\in \frakX_\infty$ by induction on $\dep(\frakx)$.
For the initial step of $\dep(\frakx) = 0$, we get $\frakx\in X\subseteq A$ and define
\begin{equation}
\deg(\frakx):= \deg_A(\frakx).
\mlabel{eq:deg1}
\end{equation}
For the inductive step of $\dep(\frakx) \geq 1$, if $\bre(\frakx) =1$, then $\frakx = \lc \lbar{\frakx}\rc$ and we define
\begin{equation}
\deg(\frakx):= \deg(\lbar{\frakx}) +1;
\mlabel{eq:deg2}
\end{equation}
if $\bre(\frakx) \geq 2$, then write $\frakx = \frakx_1 \cdots \frakx_b$ in the standard decomposition and define
\begin{equation}
\deg(\frakx):= \sum_{i=1}^b \deg(\frakx_i),
\mlabel{eq:deg3}
\end{equation}
where each $\deg(\frakx_i)$ is defined either in Eq.~(\mref{eq:deg1}) or in Eq.~(\mref{eq:deg2}) by the induction hypothesis.

\begin{remark}
For later applications, we also use the notion $\deg( c \frakx)=\deg(\frakx)$ for $c\in \bfk\backslash \{0\}$.
\mlabel{rk:deg}
\end{remark}

Denote
\begin{equation}
\hma:=\FNN(A)\,\text{ and }\, \hma_n:= \bfk \{\frakx\in \frakX_\infty \mid \deg(\frakx) \leq n\}\,\text{ for }\, n\geq 0.
\mlabel{eq:dhn}
\end{equation}
Then
\begin{equation}
\hma=\bigcup_{n\geq 0} \hma_n \,,\, \hma_0 = \bfk,\, \hma = \bfk\,1 \oplus \ker \vep_\mo \,\text{ and }\, \mop(\hma_n)\subseteq \hma_{n+1}.
\mlabel{eq:subset}
\end{equation}

Now we are going to prove that $\hma$ is a filtered bialgebra, beginning with the compatibility of the multiplication with the filtration.

\begin{lemma} For $p,q\geq 0$, we have
\begin{equation}
\hma_p \diamond \hma_q \subseteq \hma_{p+q}.
\mlabel{eq:mgrad}
\end{equation}
\mlabel{lem:mgrad}
\end{lemma}

\begin{proof}
Let $\frakx\in \hma_p$ and $\frakx'\in \hma_q$ be two basis elements in $\frakX_\infty$.
Then
$$\deg(\frakx)\leq p\,\text{ and }\, \deg(\frakx')\leq q.$$
We now verify Eq.~(\mref{eq:mgrad}) by induction on the sum $s:=p+q\geq 0$. When $s=0$, then $p=q=0$. By Eq.~(\mref{eq:subset}), we obtain that $\frakx=\frakx'=1$ and so $\frakx\,\diamond\,\frakx'=1\in \hma_0$. This finishes the initial step.

Given an $s\geq 0$, assume that Eq.~(\mref{eq:mgrad}) holds for $\frakx,\frakx'$ with $p+q=s$ and consider case $p+q=s+1$.  If $\frakx=1$ or
$\frakx'=1$, without loss of generality, letting $\frakx=1$, then $p=0$ and
$$\frakx\,\diamond\, \frakx' = \frakx' \in \hma_q = \hma_{p+q}.$$
So we may suppose $\frakx,\frakx'\neq 1$.
Write
$$\frakx =\frakx_{1} \cdots \frakx_{b} \,\text{ and }\, \frakx':=\frakx'_{1} \cdots \frakx'_{b'}\,\text{ with }\, b, b'\geq 1$$
in their standard decompositions.
Under this condition, we proceed to prove Eq.~(\mref{eq:mgrad}) by induction on the sum $t:=b+b'\geq 2$.
When $t=2$, then $b=b'=1$. If $\frakx\in X\subseteq A$ or $\frakx'\in X\subseteq A$, then
by Eq.~(\mref{eq:deg3}),
$$\frakx\,\diamond\, \frakx' = \frakx \frakx'  \in\hma_{\deg(\frakx) + \deg(\frakx')} \subseteq \hma_{p+q}.$$
It remains to check the outstanding case of
$$\frakx:=\frakx_{1}=\mop(\lbar{\frakx}_{1})\,\text{ and }\, \frakx':=\frakx'_{1}=\mop(\lbar{\frakx}'_{1}),$$
where $$\lbar{\frakx}_{1},\lbar{\frakx}'_{1}\in \frak X_\infty\,\text{ and }\,\deg(\frakx_{1})+\deg(\frakx'_{1}) \leq s+1.$$
Then
\begin{align*}
\deg(\frakx_{1})+\deg(\lbar{\frakx}'_{1})&=\deg(\frakx_{1})+\deg(\frakx'_{1})-1 \leq s,\\
\deg(\lbar{\frakx}_{1})+\deg(\frakx'_{1})&=\deg(\frakx_{1})-1+\deg(\frakx'_{1})\leq s,\\
\deg(\lbar{\frakx}_{1})+\deg(\lbar{\frakx'_{1}})&=\deg(\frakx_{1})-1+\deg(\frakx'_{1})-1\leq s-1.
\end{align*}
By the induction hypothesis on $s$, we have
\begin{align*}
\frakx_{1}\,\diamond\,\lbar{\frakx}'_{1},\, \lbar{\frakx}_{1}\,\diamond\,\frakx'_{1} \in  \hma_s\,\text{ and }\, \lbar{\frakx_{1}}\,\diamond\,\lbar{\frakx}'_{1}\in \hma_{s-1},
\end{align*}
which implies from Eq.~(\mref{eq:subset}) that
\begin{align*}
\mop(\frakx_{1}\,\diamond\,\lbar{\frakx}'_{1}), \, \mop(\lbar{\frakx}_{1}\,\diamond\,\frakx'_{1}) \in \hma_{s+1}.
\end{align*}
Hence by Eq.~(\mref{eq:shprod}),
$$\frakx_1\,\diamond\,\frakx'_{1}=\mop(\frakx_{1}\,\diamond\,\lbar{\frakx}'_{1})
+\mop(\lbar{\frakx}_{1}\,\diamond\,\frakx'_{1}) -\lambda^2 \lbar{\frakx}_{1}\,\diamond\,\lbar{\frakx'_{1}} \in \hma_{s+1}.$$

Assume that Eq.~(\mref{eq:mgrad}) holds for $b+b'=t\geq2$ and $p+q=s+1$ and consider the case when $b+b'=t+1\geq3$ and $p+q=s+1$. So either $\frakx$ or $\frakx'$ has breadth greater than or equal to $2$, giving us  three cases to consider:

\noindent
{\bf Case 1}. $\bre(\frakx)\geq2$. Let $\frakx:=\frakx_{1,\,1} \frakx_{1,\,2}$, where $\frakx_{1,\,1},\frakx_{1,\,2}\in \frak X_\infty$ with breadths $\bre(\frakx_{1,\,1}),\bre(\frakx_{1,\,2})\geq1$ respectively. By Eq.~(\mref{eq:deg3}), we obtain $\deg(\frakx)=\deg(\frakx_{1,\,1})+\deg(\frakx_{1,\,2}) $.
From Eq.~(\mref{eq:cdiam}),
$$\frakx\,\diamond\,\frakx' = (\frakx_{1,\,1} \frakx_{1,\,2})\,\diamond\,\frakx'
= \frakx_{1,\,1} (\frakx_{1,\,2}\,\diamond\,\frakx').$$
By the induction on $t$, we have
$$\frakx_{1,\,2}\,\diamond\,\frakx'\in \hma_{\deg(\frakx_{1,\,2})+\deg(\frakx')},$$
whence by Eq.~(\mref{eq:deg3}),
$$\frakx\,\diamond\,\frakx' = \frakx_{1,\,1} (\frakx_{1,\,2}\,\diamond\,\frakx')
\in \hma_{\deg(\frakx_{1,\,1}) + \deg(\frakx_{1,\,2})+\deg(\frakx')} = \hma_{\deg(\frakx) +\deg(\frakx')}.$$

\noindent
{\bf Case 2}. $\bre(\frakx')\geq2$. The proof of this case is similar to Case 1.

\noindent
{\bf Case 3}. $\bre(\frakx)\geq2$ and $\bre(\frakx')\geq2$. Let $\frakx:=\frakx_{1,\,1}\frakx_{1,\,2}$ and $\frakx':=\frakx'_{1,\,1} \frakx'_{1,\,2}$,
where $\frakx_{1,\,1},\frakx_{1,\,2},
\frakx'_{1,\,1},\frakx'_{1,\,2}\in \frak X_\infty$ with breadths $\bre(\frakx_{1,\,1}),\bre(\frakx_{1,\,2}),\bre(\frakx'_{1,\,1}),\bre(\frakx'_{1,\,2})\geq1$ respectively. By Eq.~(\mref{eq:deg3}), we obtain
\begin{equation*}
\deg(\frakx)=\deg(\frakx_{1,\,1})+\deg(\frakx_{1,\,2})\,\text{ and }\,
\deg(\frakx')=\deg(\frakx'_{1,\,1})+\deg(\frakx'_{1,\,2}).
\mlabel{eq:c0}
\end{equation*}
Thus by Eq.~(\mref{eq:cdiam}),
$$\frakx\,\diamond\,\frakx' =(\frakx_{1,\,1} \frakx_{1,\,2})\,\diamond\,(\frakx'_{1,\,1} \frakx'_{1,\,2})
=\frakx_{1,\,1} (\frakx_{1,\,2}\,\diamond\,\frakx'_{1,\,1}) \frakx'_{1,\,2}.$$
By the induction on $t$, we have
\begin{equation*}
\frakx_{1,\,2}\,\diamond\,\frakx'_{1,\,1}\in \hma_{\deg(\frakx_{1,\,2})+\deg(\frakx'_{1,\,1})}.
\mlabel{eq:c2}
\end{equation*}
With a similar argument to Case 1. we get
$$\frakx\,\diamond\,\frakx' = \frakx_{1,\,1} (\frakx_{1,\,2}\,\diamond\,\frakx'_{1,\,1}) \frakx'_{1,\,2} \in \hma_{\deg(\frakx)+\deg(\frakx')}.$$
This finishes the proof.
\end{proof}

For the compatibility of the coproduct with the filtration, we have
\begin{lemma}
For $n\geq0$, we have
\begin{equation}
\Delta_\mo(\hma_n)\subseteq\sum\limits^{}_{p+q=n}\hma_p\ot \hma_q.
\mlabel{eq:cmgrad}
\end{equation}
\mlabel{lem:cmgrad}
\end{lemma}

\begin{proof}
We verify Eq.~(\mref{eq:cmgrad}) by showing

\begin{claim}
For any $\frakx\in \frakX_\infty$, we have
\begin{equation}
\Delta_\mo(\frakx)=\sum_{(\frakx)} \frakx_{(1)}\ot \frakx_{(2)},
\mlabel{eq:cmgrad1}
\end{equation}
where $\frakx_{(1)}$ and $\frakx_{(2)}$ are non-zero linear multiples of elements of $\frakX_\infty$ with $\deg(\frakx_{(1)})+\deg(\frakx_{(2)})\leq \deg(\frakx)$. Here we have adapted the notation in Remark~\mref{rk:deg}.
\mlabel{cl:copfil}
\end{claim}

To prove this claim we proceed by induction on $\deg(\frakx)\geq 0$. For the initial step of $\deg(\frakx)=0$,
we get $\frakx=1$ and the result holds.
Assume that Claim~(\mref{cl:copfil}) holds for $\frakx\in \hma_k$ and consider $\frakx\in \hma_{k+1}$ for some $k\geq 0$.

In this case, we prove Claim~(\mref{cl:copfil}) by induction on the breadth $b:=\bre( \frakx)\geq 1$. If $b=1$, we have $\frakx\in X\subseteq A$ or $ \frakx=\mop(\lbar{ \frakx})$ for some $\lbar{ \frakx}\in \frak X_\infty$. For the former, Claim~(\mref{cl:copfil}) holds since $\Delta_\mo$ is given by $\Delta_A$ and $A$ is a connected filtered bialgebra by our hypothesis. For the latter, applying the induction hypothesis on $n$, we can write
$$\Delta_\mo(\lbar{ \frakx})=\sum_{(\lbar{ \frakx})}\,\lbar{ \frakx}_{(1)}\ot\lbar{ \frakx}_{(2)},$$
where $\deg(\lbar{ \frakx}_{(1)})+\deg(\lbar{ \frakx}_{(2)}) \leq \deg(\lbar{ \frakx})=k$, with the notion in Remark~\ref{rk:deg}.
By Eq.~(\mref{eq:cocyc}), we have
\begin{eqnarray*}
\Delta_\mo( \frakx)=\Delta_\mo(\mop(\lbar{ \frakx}))
&=& \frakx\ot 1+(\id\ot \mop)\Delta_\mo(\lbar{ \frakx}) +\lambda \lbar{\frakx}\ot 1\\
&=& \frakx\ot 1+\sum_{(\lbar{ \frakx})}\,(\lbar{ \frakx}_{(1)}\ot \mop(\lbar{ \frakx}_{(2)}) +
\lambda \lbar{\frakx}\ot 1.
\end{eqnarray*}
By Eq.~(\mref{eq:dhn}), it is sufficient to show that the sum of degrees of tensor factors in each summand is less than or equal to $k+1$,
which follows from
\begin{eqnarray*}
&&\deg( \frakx)+\deg(1)= \deg( \frakx) \leq k+1,\, \deg( \lbar{\frakx})+\deg(1) = \deg( \lbar{\frakx}) \leq  k,\\
&&\deg(\lbar{ \frakx}_{(1)})+\deg(\mop(\lbar{ \frakx}_{(2)}))
=\deg(\lbar{ \frakx}_{(1)})+\deg(\lbar{ \frakx}_{(2)})+1 \leq k+1.
\end{eqnarray*}

Assume that Claim~(\mref{cl:copfil}) holds for $ \frakx\in \hma_{k+1}$ with $\bre( \frakx)=b$ and consider the case $ \frakx\in \hma_{k+1}$ with $\bre( \frakx)=b+1\geq2$.
Let $ \frakx= \frakx_{1}  \frakx_{2}$, where $ \frakx_{1}, \frakx_{2}\in \frak X_\infty$ with $\bre( \frakx_{1}),\bre( \frakx_{2})\geq1$.
From Eq.~(\mref{eq:deg3}), we have
\begin{equation}
\deg( \frakx_{1})+\deg( \frakx_{2})=\deg( \frakx) \leq k+1.
\mlabel{eq:Num2}
\end{equation}
Write
\begin{equation*}
\begin{aligned}
\Delta_\mo( \frakx_{1})=\sum_{( \frakx_{1})} \frakx_{1(1)}\ot \frakx_{1(2)}\,\text{ and }\, \Delta_\mo( \frakx_{2})=\sum_{( \frakx_{2})} \frakx_{2(1)}\ot \frakx_{2(2)}.
\end{aligned}
\end{equation*}
By the induction hypothesis on $b$, we have
\begin{equation}
\deg( \frakx_{1(1)}) + \deg( \frakx_{1(2)}) \leq \deg( \frakx_{1})\,\text{ and }\,
\deg( \frakx_{2(1)}) + \deg( \frakx_{2(2)}) \leq \deg( \frakx_{2}).
\mlabel{eq:gapnum1}
\end{equation}
So we have
\begin{align*}
\Delta_\mo( \frakx)&=\Delta_\mo( \frakx_{1}  \frakx_{2})
=\Delta_\mo( \frakx_{1}\,\diamond\, \frakx_{2})
=\Delta_\mo( \frakx_{1})\,\diamond\, \Delta_\mo( \frakx_{2})\\
&=\left(\sum_{( \frakx_{1})} \frakx_{1(1)}\ot \frakx_{1(2)}\right)\,\diamond\,
\left(\sum_{( \frakx_{2})} \frakx_{2(1)}\ot \frakx_{2(2)}\right)\\
&=\sum_{( \frakx_{1})}\sum_{( \frakx_{2})} ( \frakx_{1(1)}\,\diamond\,\frakx_{2(1)})\ot
( \frakx_{1(2)}\,\diamond\, \frakx_{2(2)}).
\end{align*}
By Eq.~(\mref{eq:mgrad}),
$$ \frakx_{1(1)}\,\diamond\,\frakx_{2(1)} \in \hma_{\deg(\frakx_{1(1)}) + \deg(\frakx_{2(1)})}
\,\text{ and }\,  \frakx_{1(2)}\,\diamond\, \frakx_{2(2)} \in \hma_{\deg(\frakx_{1(2)}) + \deg(\frakx_{2(2)})},$$
which implies from Eqs.~(\mref{eq:dhn}), (\mref{eq:Num2}) and~(\mref{eq:gapnum1}) that Claim~\mref{cl:copfil} holds.
\end{proof}

We now arrive at our last main result.

\begin{theorem}
Let $A=\cup_{n\geq 0}A_n$ be a connected filtered bialgebra with a filtered basis. Then $\hma = \FNN(A)$ is also a connected filtered bialgebra, and hence a Hopf algebra.
\mlabel{thm:hopf}
\end{theorem}

\begin{proof}
By Lemma~\mref{lem:hopf}, we just need to prove that $\FNN(A)$ is a connected filtered bialgebra. This follows from Lemmas~\mref{lem:mgrad}, \mref{lem:cmgrad} and Eq.~(\mref{eq:subset}).
\end{proof}

\noindent {\bf Acknowledgements}: This work was supported by the National Natural Science Foundation of
China (No.~11771190), the
Fundamental Research Funds for the Central Universities (No.~lzujbky-2017-162) and the Natural Science Foundation of Gansu Province (No.~17JR5RA175).
The authors thank the referees for helpful suggestions.

\end{document}